\documentclass{amsart}
\usepackage{graphicx} 
\usepackage{yufei}

\title{Hypergraph independence polynomials with a zero close to the origin}
\author{Shengtong Zhang}

\address{Department of Mathematics, Stanford University, Stanford, CA 94305, USA}
\email{stzh1555@stanford.edu}

\date{December 2024}

\begin{document}

\maketitle

\begin{abstract}
For each uniformity $k \geq 3$, we construct $k$-uniform linear hypergraphs $G$ with arbitrarily large maximum degree $\Delta$ whose independence polynomial $Z_G$ has a zero $\lambda$ with $\abs{\lambda} = O\left(\frac{\log \Delta}{\Delta}\right)$. This disproves a recent conjecture of Galvin, McKinley, Perkins, Sarantis, and Tetali.
\end{abstract}

\section{Introduction}
A \textbf{hypergraph} $G = (V, E)$ is a set of vertices $V$ together with a set of edges $E \subset 2^V$. A hypergraph is \textbf{$k$-uniform} if every edge has size $k$. The degree of a vertex $v \in V$, denoted by $d(v)$, is the number of edges it appears in; in a hypergraph with maximum degree $\Delta$, each vertex appears in at most $\Delta$ edges.

An \textbf{independent set} in $G$ is a set of vertices $I \subset V$ such that $I$ contains no edge. Let $\cI(G)$ denote the family of all independent sets in $G$. The \textbf{independence polynomial} of $G$ is defined by
$$Z_G(\lambda) = \sum_{I \in \cI(G)} \lambda^{\abs{I}}.$$
This polynomial plays an important role in mathematics, physics and computer science \cite{DP21, HL70, SSPV22, SS05, S85, S10, W06}. A key property for understanding this polynomial is the largest radius of a disk-shaped \textbf{zero-free region(ZFR)}, a region in $\CC$ where $Z_G$ has no zero. We refer the reader to the introduction of \cite{GMPST22} for a survey of how knowledge of the zeros of $Z_G$ can lead to interesting results about independent set.

When $G$ is a graph with maximum degree $\Delta$, the ZFRs for $Z_G$ are well-understood \cite{PR18, S85}. Specifically, Shearer \cite{S85} showed that $Z_G$ has no zero inside the disk $\abs{\lambda} < \frac{(\Delta - 1)^{\Delta - 1}}{\Delta^{\Delta}}$, and this bound is the best possible.

In a recent paper \cite{GMPST22}, Galvin, McKinley, Perkins, Sarantis and Tetali studied the zeros of $Z_G$ when $G$ is a general hypergraph with given maximum degree $\Delta$. They showed that $Z_G$ has no zero inside the disk $\abs{\lambda} < \frac{\Delta^\Delta}{(\Delta + 1)^{(\Delta + 1)}}$. Furthermore, for each uniformity $k \geq 2$, they constructed a family of $k$-uniform hypergraphs with arbitrarily large maximum degree $\Delta$ such that $Z_G$ has a zero $\lambda$ with $\abs{\lambda} < O_k\left(\frac{\log \Delta}{\Delta}\right)$, thereby showing that their bound is tight up to a logarithmic factor if we place no additional assumption on $G$.

A hypergraph is \textbf{linear} if each pair of edges intersect in at most one vertex. The aforementioned constructions in \cite[Section 4]{GMPST22} are far from linear, since they contain edges that intersect in $(k - 1)$ vertices. In \cite[Conjecture 3]{GMPST22}, Galvin, McKinley, Perkins, Sarantis and Tetali conjectured that their lower bound on the maximum radius of the zero-free disk for $Z_G$ can be improved under the additional assumption that $G$ is linear. 
\begin{conjecture}
\label{conj:main}
For each $k \geq 2$, there exists a constant $C_k > 0$ such that the following is true. If $G$ is
a $k$-uniform, linear hypergraph with maximum degree $\Delta$ and if
$$\abs{\lambda} \leq C_k \Delta^{-\frac{1}{k - 1}}$$
then $Z_G(\lambda) \neq 0$.
\end{conjecture}
This conjecture is motivated by results on asymptotic enumeration in \cite{MNPS20}. Galvin, McKinley, Perkins, Sarantis and Tetali verified this conjecture when $G$ is a hypertree \cite[Theorem 4]{GMPST22}, 

In this note, we disprove \cref{conj:main} in any uniformity $k \geq 3$. Our counterexample shows that the radius of the disk-shaped ZFR $\abs{\lambda} < \frac{\Delta^\Delta}{(\Delta + 1)^{(\Delta + 1)}}$ is tight up to a logarithmic factor even if we assume that $G$ is a $k$-uniform linear hypergraph.
\begin{theorem}
    \label{thm:main}
    For each uniformity $k \geq 3$ and $\Delta > 100k^2$, there exists a $k$-uniform linear hypergraph $G$ with maximum degree $\Delta$, such that $Z_{G}$ has a negative real zero $\lambda$ with $\lambda \in [-\frac{6k \log \Delta}{\Delta}, 0]$.
\end{theorem}
\section{The Counterexample}
We begin by describing a general construction.
\begin{definition}
Let $G = (V, E)$ be a hypergraph. We define $S_G$ as the hypergraph whose vertex set is $V \sqcup E$, and whose edge set is $\{e \cup\{e\}: e \in E\}$.
\end{definition}
\usetikzlibrary{topaths,calc}
\begin{figure}[h]
\centering
\begin{subfigure}[b]{0.45\linewidth}
\centering
\begin{tikzpicture}[label distance=0.2mm]
    \node (v1) at (2,3.5) {};
    \node (v2) at (0,0) {};
    \node (v3) at (4,0) {};
    \foreach \v in {1,2,3} {
        \fill (v\v) circle (0.1);
    };
    \draw (v1) -- (v2) node[midway,label=above left:$e_1$] {};
    \draw (v1) -- (v3) node[midway,label=above right:$e_2$] {};
    \draw (v2) -- (v3) node[midway,label=below:$e_3$] {};
    \fill (v1) circle (0.1) node [label=above:$v_1$] {};
    \fill (v2) circle (0.1) node [label=below:$v_2$] {};
    \fill (v3) circle (0.1) node [label=below:$v_3$] {};
\end{tikzpicture}
\caption{A graph $G$.}
\end{subfigure}
\begin{subfigure}[b]{0.45\linewidth}
\centering
\begin{tikzpicture}[scale=0.50]
    \node (v1) at (2,3.5) {};
    \node (v2) at (0,0) {};
    \node (v3) at (4,0) {};
    \node (v4) at (-2, -3.5) {};
    \node (v5) at (8,0) {};
    \node (v6) at (6, -3.5) {};

    \begin{scope}[fill opacity=0.8]
    \filldraw[fill=white!0,rounded corners=0.25cm,rotate=60] ($(v4)+(-0.5,-0.5)$) rectangle ($(v1)+(0.5,0.5)$);
    \filldraw[fill=white!0,rounded corners=0.25cm,rotate=120] ($(v6)+(-0.5,-0.5)$) rectangle ($(v1)+(0.5,0.5)$);
    \filldraw[fill=white!0,rounded corners=0.25cm] ($(v2)+(-0.5,-0.5)$) rectangle ($(v5)+(0.5,0.5)$);
    \end{scope}

    \foreach \v in {1,2,...,6} {
        \fill (v\v) circle (0.1);
    }

    \fill (v1) circle (0.1) node [label={[label distance=-1.5mm]-90:$v_1$}] {};
    \fill (v2) circle (0.1) node [right] {$v_2$};
    \fill (v3) circle (0.1) node [right] {$v_3$};
    \fill (v4) circle (0.1) node [label={[label distance=-1.5mm]90:$e_1$}] {};
    \fill (v5) circle (0.1) node [left] {$e_3$};
    \fill (v6) circle (0.1) node [above] {$e_2$};
\end{tikzpicture}
\caption{The $3$-uniform hypergraph $S_G$.}
\end{subfigure}
\end{figure}
We begin by showing a basic property of $S_G$.
\begin{proposition}
\label{prop:basic}
If $G = (V, E)$ is a $(k - 1)$-uniform linear hypergraph with maximum degree $\Delta$, then $S_G$ is a $k$-uniform linear hypergraph with maximum degree $\Delta$. 
\end{proposition}
\begin{proof}
Each edge in $S_G$ has the form $e \cup\{e\}$ for some $e \in E$, and $\abs{e \cup\{e\}} = \abs{e} + 1 = k$. So $S_G$ is $k$-uniform. 

For any pair of distinct edges $e_1 \cup \{e_1\}$ and $e_2 \cup \{e_2\}$ in $S_G$, we have
$$\abs{(e_1 \cup \{e_1\}) \cap (e_2 \cup \{e_2\}) } = \abs{e_1 \cap e_2} \leq 1.$$
Thus $S_G$ is linear.

Finally, each vertex in $V$ has the same degree in $G$ and $S_G$, while each element in $E$ has degree $1$ in $S_G$. Therefore, the maximum degree of $S_G$ is the same as the maximum degree of $G$.
\end{proof}
In the next two lemmas, we give an explicit formula for the independence polynomial $Z_{S_G}$ of $S_G$ and prove that it has a zero close to the origin whenever $G$ satisfies a mild expansion property.
\begin{lemma}
\label{lem:indep-poly-formula}
Let $G = (V, E)$ be a hypergraph. For each set of vertices $S \subset V$, let $E(S)$ denote the edges of $G$ with at least one vertex in $S$, and write $e(S) = \abs{E(S)}$. Then we have
$$Z_{S_G}(\lambda) = \sum_{S \subset V} \lambda^{\abs{V} - \abs{S}} (1 + \lambda)^{e(S)}.$$
\end{lemma}
\begin{proof}
Classifying the independence sets of $S_G$ based on their intersections with $V \subset V(S_G)$, we have
$$Z_{S_G}(\lambda) = \sum_{S \subset V} \sum_{I \in \cI(S_G): I \cap V = V \backslash S} \lambda^{\abs{I}}.$$
A set of vertices $I \subset V \sqcup E$ with $I \cap V = V \backslash S$ is independent in $S_G$ if and only if $J := I \cap E$ is contained in $E(S)$. So we have
$$\sum_{I \in \cI(S_G): I \cap V = V \backslash S} \lambda^{\abs{I}} = \lambda^{\abs{V \backslash S}} \sum_{J \subset E(S)} \lambda^{\abs{J}} = \lambda^{\abs{V \backslash S}} (1 + \lambda)^{e(S)}$$
and the lemma follows.
\end{proof}
\begin{lemma}
\label{lem:hypergraph-ZFR}
Let $G = (V, E)$ be a hypergraph with $n \geq 3$ vertices. Assume that $n$ is odd. Furthermore, assume that for some $\alpha \in [3\log n, n]$, we have $e(S) \geq \alpha\abs{S}$ for any $S \subset V$. Then $Z_{S_G}$ has a negative real zero in the interval 
$$\left[-\frac{3 \log n}{\alpha}, 0\right].$$
\end{lemma}
\begin{proof}
    Set $\lambda_0 = -\frac{3\log n}{\alpha}$. As $Z_{S_G}(0) = 1$, it suffices to show that $Z_{S_G}(\lambda_0) < 0$.

    By \cref{lem:indep-poly-formula}, we have the identity
    $$Z_{S_G}(\lambda_0) = \sum_{S \subset V} \lambda_0^{\abs{V} - \abs{S}} (1 + \lambda_0)^{e(S)}.$$
    We isolate the term with $S = \emptyset$ and obtain
    $$Z_{S_G}(\lambda_0) = \lambda_0^{\abs{V}} \left(1 + \sum_{S \subset V, S \neq \emptyset} \lambda_0^{-\abs{S}} (1 + \lambda_0)^{e(S)}\right)$$
    Since $0 \leq 1 + \lambda_0 \leq e^{\lambda_0}$, for each $S \subset V$ we can estimate
    $$\abs{\lambda_0^{-\abs{S}} (1 + \lambda_0)^{e(S)}} \leq \abs{\lambda_0}^{-\abs{S}} e^{\lambda_0 e(S)} \leq \alpha^{\abs{S}} e^{-3 \log n \cdot \abs{S}}.$$
    Thus we have
    $$\abs{\sum_{S \subset V, S \neq \emptyset} \lambda_0^{-\abs{S}} (1 + \lambda_0)^{e(S)}} \leq \sum_{k = 1}^n \binom{n}{k} \alpha^{k} e^{-3 \log n \cdot k} \leq \sum_{k = 1}^n \frac{1}{k!} \left(\frac{\alpha}{n^2}\right)^k.$$
    where the second inequality uses the estimate $\binom{n}{k} \leq \frac{n^k}{k!}$.
    
    We assume that $\alpha \leq n$, so $\left(\frac{\alpha}{n^2}\right)^k < \frac{1}{n}$ for each $k \geq 1$. This leads to
    $$\abs{\sum_{S \subset V, S \neq \emptyset} \lambda_0^{-\abs{S}} (1 + \lambda_0)^{e(S)}} \leq \frac{1}{n}\sum_{k = 1}^n \frac{1}{k!} < \frac{e}{n} < 1.$$
    Thus we conclude that
    $$1 + \sum_{S \subset V, S \neq \emptyset} \lambda_0^{-\abs{S}} (1 + \lambda_0)^{e(S)} > 0$$
    so $Z_{S_G}(\lambda_0) < 0$, as desired.
\end{proof}
Our main theorem is an easy corollary of this result. Indeed, for any $(k - 1)$-uniform hypergraph $G$ we can show that $e(S) \geq \frac{\delta(G) \abs{S}}{k - 1}$ for any $S \subset V$. So when $\delta(G) \geq \Delta(G) - 1$, one can take $\alpha = \frac{\Delta(G) - 1}{k - 1}$ in \cref{lem:hypergraph-ZFR}. We give an explicit construction of such a hypergraph.
\begin{lemma}
    \label{lem:hypergraph-construction}
    For any uniformity $k \geq 2$ and $\Delta \geq k$, there exists a $k$-uniform, $\Delta$-regular linear hypergraph $H_{k, \Delta}$ on at most $2k \Delta$ vertices.
\end{lemma}
\begin{proof}
    We take a prime $p$ in $[\Delta, 2\Delta]$, which exists by Chebyshev's theorem. Let $H_{k, \Delta}$ be the hypergraph on the vertex set $V = [k] \times \ZZ_p$ with an edge $\{(i, a + id): i \in [k]\}$ for each pair $(a, d) \in \ZZ_p \times [\Delta]$. The number of vertices in $H_{k, \Delta}$ is $\abs{V} = kp \leq 2 \Delta k$.

    Any vertex $(i, x) \in V$ is contained precisely in the edges corresponding to $(a, d) \in \ZZ_p \times [\Delta]$ with $a \equiv x - id \pmod{p}$. As there is exactly one $a \in \ZZ_p$ corresponding to each $d \in [\Delta]$, $H_{k, \Delta}$ is $\Delta$-regular.
    
    The size of the intersection between two distinct edges corresponding to $(a, d)$ and $(a', d')$ is the number of solutions $i \in [k]$ to the linear congruence equation $a + id \equiv a' + id' \pmod{p}$. As $k \leq \Delta \leq p$, there is at most one solution $i \in [k]$ to this linear congruence equation, so every pair of hyperedges in $H_{k, \Delta}$ intersect in at most one vertex. Therefore, $H_{k, \Delta}$ is a linear hypergraph.
\end{proof}
\begin{proof}[Proof of \cref{thm:main}]
Let $H = H_{(k - 1), \Delta}$ be the $(k - 1)$-uniform linear hypergraph constructed in \cref{lem:hypergraph-construction}. If $H$ has an even number of vertices, we remove an arbitrary vertex $v$ of $H$ together with any edge containing the vertex. Thus, we obtain a $(k - 1)$-uniform linear hypergraph $H$ with an odd number of vertices, maximum degree $\Delta$, and minimum degree at least $(\Delta - 1)$. By \cref{prop:basic}, $G = S_H$ is a $k$-uniform linear hypergraph with maximum degree $\Delta$.

Let $n \leq 2k \Delta$ be the number of vertices in $H$. For any vertex subset $S$ of $H$, the $(k - 1)$-uniformity of $H$ implies that the number of edges in $G$ with at least one vertex in $S$ is lower bounded by
$$e(S) \geq \frac{1}{k - 1} \sum_{v \in S} d(v) \geq \frac{\Delta - 1}{k - 1} \abs{S}.$$
By our assumptions $\Delta > 100k^2$ and $n \leq 2k \Delta$, we can check that 
$$\frac{\Delta - 1}{k - 1} > \frac{\Delta}{k} > 10 \sqrt{\Delta} > 10 \log \Delta > 3 \log(2k\Delta) \geq 3 \log(n).$$
Furthermore, we have $\frac{\Delta - 1}{k - 1} \leq \Delta - 1 < n$. So we can apply \cref{lem:hypergraph-ZFR} with $\alpha = \frac{\Delta - 1}{k - 1}$. We conclude that $Z_{G}$ has a negative real zero $\lambda$ with 
$$-\lambda \leq \frac{3\log n}{(\Delta - 1) / (k - 1)} \leq \frac{3 k \log(2\Delta k)}{\Delta} \leq \frac{6k \log \Delta}{\Delta}$$
where the last two inequalities follow from $\Delta > 100k^2$. So $G$ satisfies the requirements of \cref{thm:main}.
\end{proof}
\section*{Acknowledgement}
I am supported by the Craig Franklin Fellowship in Mathematics at Stanford University. I thank Professor Nima Anari and Professor Will Perkins for many helpful discussions and comments. I also thank the anonymous referee for valuable feedbacks on the manuscript.

\bibliographystyle{plain}
\bibliography{bib.bib}

\end{document}